\documentclass[11pt]{article}
\pdfoutput=1
\usepackage[utf8]{inputenc}
\usepackage[finnish,english]{babel}					
\usepackage{amscd,amssymb,amsfonts}

\usepackage[utf8]{inputenc} 
\usepackage[T1]{fontenc}    
\usepackage{hyperref}       
\usepackage{url}            
\usepackage{booktabs}       
\usepackage{nicefrac}       
\usepackage{microtype}      
\usepackage{lipsum}
\usepackage{amsmath}
\usepackage{graphicx}
\usepackage{color}
\usepackage{hyperref}
\usepackage{subcaption}
\usepackage{caption}
\usepackage{multirow}
\usepackage{epstopdf}
\usepackage{appendix}
\usepackage{pdfpages}
\usepackage{float}
\usepackage[ruled,vlined]{algorithm2e}
\usepackage{amsthm}

\newtheorem{theorem}{Theorem}[section]

\newtheorem{lemma}[theorem]{Lemma}
\newtheorem{remark}{Remark}

\newcommand{\argmin}{\mathop{\mbox{arg\,min}}}

\def\bra{\langle}
\def\ket{\rangle}
\def\e{\epsilon}
\def \ba {\begin {eqnarray*} }
\def \ea {\end {eqnarray*} }
\def \beq {\begin {eqnarray}}
\def \eeq {\end {eqnarray}}

\newcommand{\gvec}{\mathbf{g}}
\newcommand{\svec}{\mathbf{s}}
\newcommand{\evec}{\mathbf{e}}
\newcommand{\vvec}{\mathbf{v}}
\newcommand{\mvec}{\mathbf{m}}
\newcommand{\A}{\mathcal{A}}

\newcommand{\R}{\mathbb{R}}
\newcommand{\Regul}{\mathcal{R}}
\newcommand{\Sep}{\mathcal{S}}
\newcommand{\Radon}{A}

\title{Material-separating regularizer for \\multi-energy X-ray tomography }

\author{Jacek Gondzio, Matti Lassas,\\ Salla-Maaria Latva-Äijö, Samuli Siltanen, Filippo Zanetti}

\begin{document}

\maketitle

\begin{abstract}
    Dual-energy X-ray tomography is considered in a context where the target under imaging consists of two distinct materials. The materials are assumed to be possibly intertwined in space, but at any given location there is only one material present. Further, two X-ray energies are chosen so that there is a clear difference in the spectral dependence of the attenuation coefficients of the two materials. A novel regularizer is presented for the inverse problem of reconstructing separate tomographic images for the two materials. A combination of two things, (a) non-negativity constraint, and (b) penalty term containing the inner product between the two material images, promotes the presence of at most one material in a given pixel. A preconditioned interior point method is derived for the minimization of the regularization functional. Numerical tests with digital phantoms suggest that the new algorithm outperforms the baseline method, Joint Total Variation regularization, in terms of correctly material-characterized pixels. While the method is tested only in a two-dimensional setting with two materials and two energies, the approach readily generalizes to three dimensions and more materials. The number of materials just needs to match the number of energies used in imaging. 
\end{abstract}



\section{Introduction}

Consider a physical object consisting of two different materials. It might be a machine part manufactured as a metal-plastic composite, or a fragile cultural heritage object unearthed at an archaeological site, or a two-phase fluid flow inside a process industry pipeline at a given time instant. We are interested in using X-ray tomography as a means of nondestructive testing to find out how the two materials are intertwined. To this end, we introduce a {\it novel regularization method} for dual-energy X-ray tomography for material decomposition and propose a {\it specialized interior point method} to solve the underlying optimization problem. 

We restrict here to the intersection of the object with a  two-dimensional square $\Omega\subset\R^2$. The measured X-rays thus are assumed to travel in the plane determined by $\Omega$; one can then stack several 2D reconstructions to achieve a 3D reconstruction. This restriction is only for simplicity of exposition and computation; our methods do generalize to higher dimensions.

We discretize $\Omega$  into $N\times N$ square-shaped pixels. 
There are two unknowns: non-negative $N\times N$ matrices  $G^{(1)}$ and $G^{(2)}$ modelling the distributions of material 1 and material 2, respectively. The number $G^{(\ell)}_{i,j}\geq 0$ represents the concentration of material $\ell$ in pixel $(i,j)$, where $i$ is row index and $j$ is column index. In numerical computations we represent the elements of the pair of material matrices $(G^{(1)},G^{(2)})\in(\R^{N\times N})^2$, as a vertical vector 
$$
\gvec = \left[\!\!\begin{array}{l}\gvec^{(1)}\\\gvec^{(2)}\end{array}\!\!\right]\in \R^{2N^2}.
$$

We consider recording X-ray transmission data with two different energies, low and high, resulting in two $M$-dimensional data vectors called $ \mvec^L$ and $\mvec^H$. The low-energy measurement is given by
\begin{equation}\label{measmodel_low}
    \mvec^L = c_{11}\Radon^L\gvec^{(1)} + c_{12}\Radon^L\gvec^{(2)},
\end{equation}
as both materials attenuate the low-energy X-rays with individual strengths described by the constants $c_{11}>0$ and $c_{12}>0$. Note that empirical values of $c_{11}$ and $c_{12}$ can be found by measuring pure samples of each of the two known materials. The $M{\times}N^2$ matrix $A^L$ encodes the geometry of the tomographic measurement in a standard way \cite[Section 2.3.4]{mueller2012linear}; it contains path lengths of X-rays traveling inside the pixels in $\Omega$. We have $M = r_0P$ with $P$ the number of projection directions and $r_0$ the amount of detector elements in the one-dimensional line camera.

Analogously we get for the high-energy measurement
\begin{equation}\label{measmodel_high}
    \mvec^H = c_{21}\Radon^H\gvec^{(1)} + c_{22}\Radon^H\gvec^{(2)},
\end{equation}
where the geometric system matrix $\Radon^H$ is possibly different from $\Radon^L$. See Figure \ref{fig:arrangement} for examples of imaging geometries. Again, $c_{21}>0$ and $c_{22}>0$ can be determined empirically. 

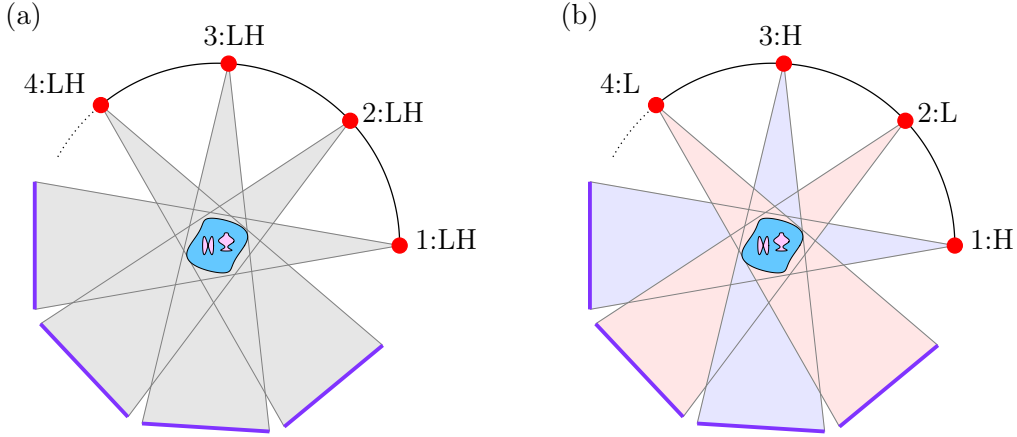
\begin{figure}
    \begin{picture}(200,170)
    \put(10,10){\includegraphics[width=5cm]{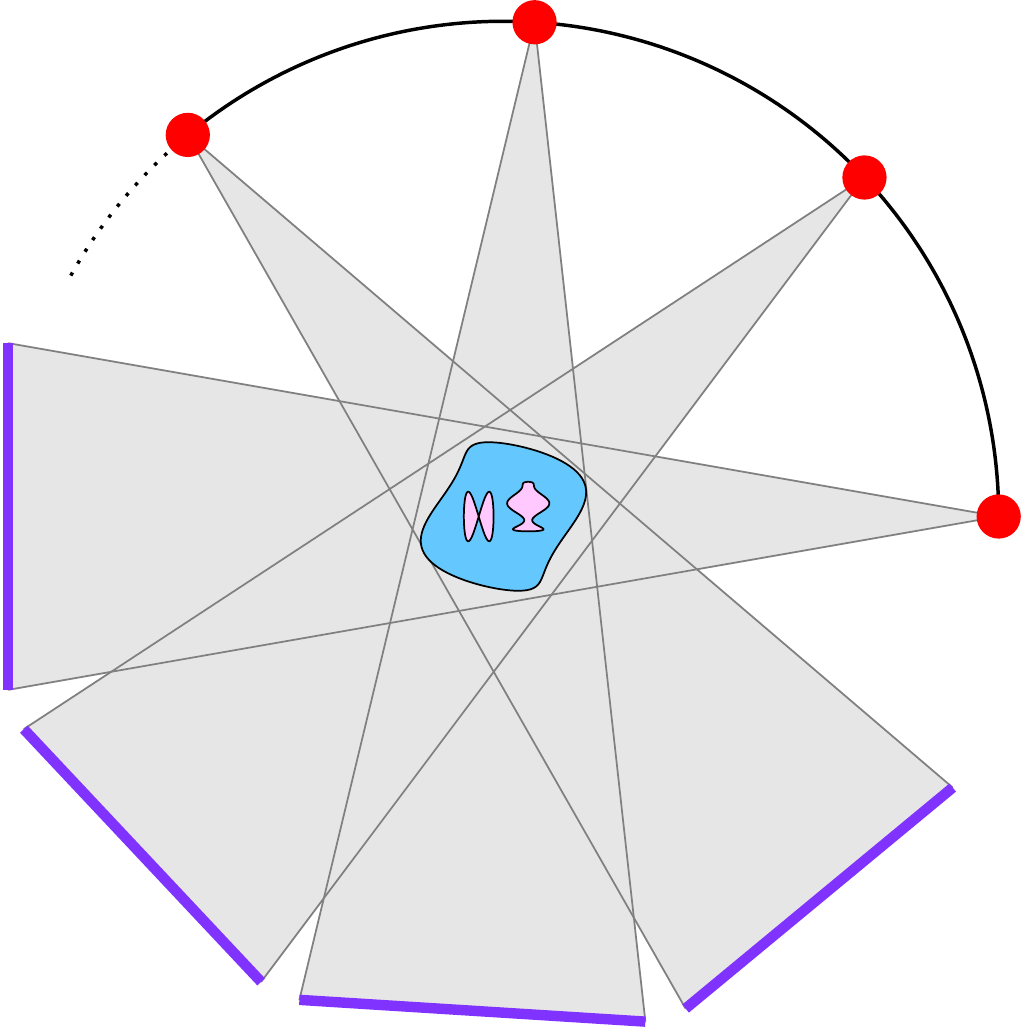}}
    \put(0,165){(a)}
    \put(155,80){1:LH}
    \put(135,128){2:LH}
    \put(75,157){3:LH}
    \put(7,139){4:LH}
    \put(220,10){\includegraphics[width=5cm]{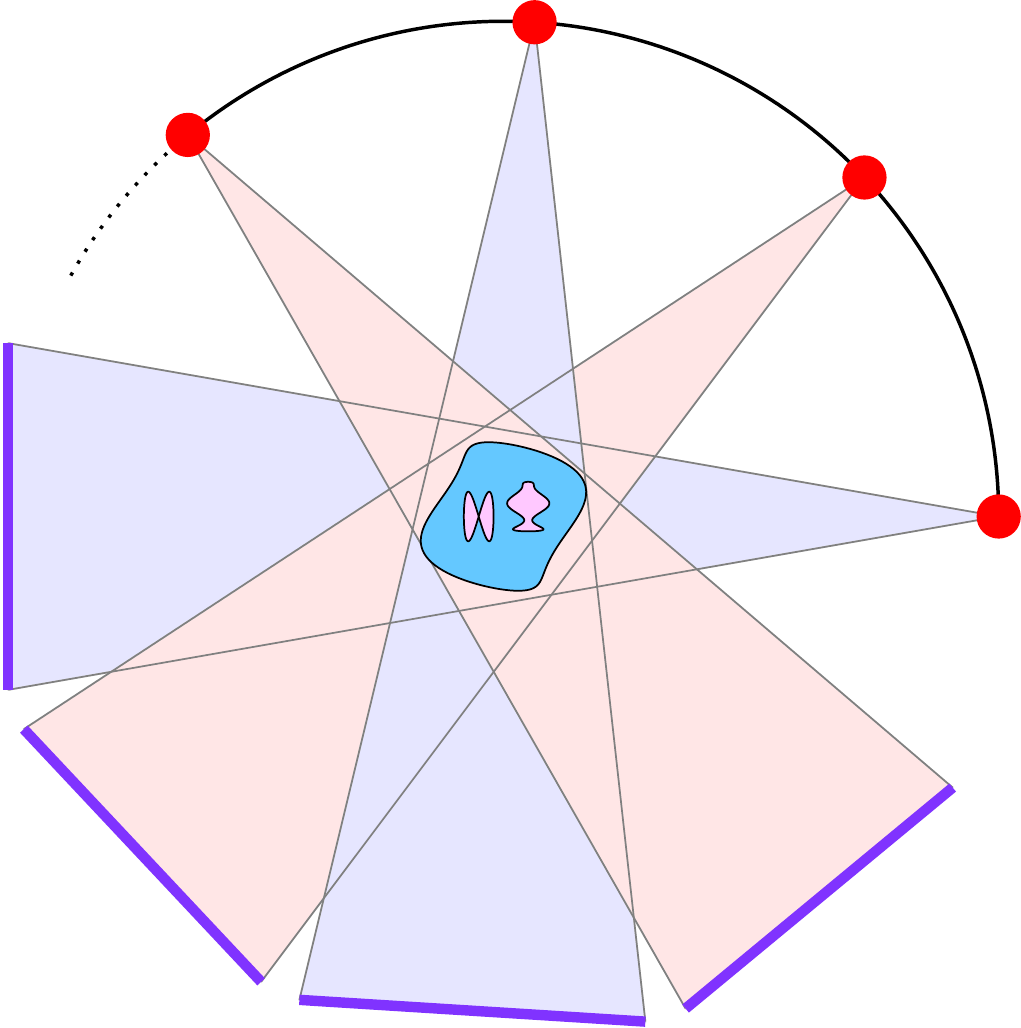}}
    \put(210,165){(b)}
    \put(365,80){1:H}
    \put(345,128){2:L}
    \put(285,157){3:H}
    \put(225,139){4:L}
    \end{picture}
    \caption{Alternative imaging protocols. (a) Two projection images are recorded from each source location: one with low (L) and another with high (H) energy. In this case we have  $\Radon^H=\Radon^L$. (b) Only one projection image is recorded at every source location, alternating between low and high energies. In this case we have  $\Radon^H\not=\Radon^L$.}
    \label{fig:arrangement}
\end{figure}

Now we can combine both measurements in a unified linear system:
\begin{equation}\label{unifiedsystem}
    \mvec=\left[\!\!\begin{array}{l}\mvec^L\\\mvec^H\end{array}\!\!\right]
    =\begin{bmatrix}
c_{11}\Radon^L &c_{12}\Radon^L\\
c_{21}\Radon^H &c_{22}\Radon^H\\
\end{bmatrix}
\left[\!\!\begin{array}{l}\gvec^{(1)}\\\gvec^{(2)}\end{array}\!\!\right]=\A \gvec.
\end{equation}
The core idea in dual-energy X-ray tomography for material decomposition is to choose the two energies so that the two materials respond to them differently. For example, one material might be quite indifferent to the energy change, while the other could attenuate very differently according to energy. Then the solution of (\ref{unifiedsystem}) is rather analogous to solving a system of two linear equations for two variables.

We propose a novel variational regularization approach in the space $\gvec\in \R^{2N^2}$, including a non-negativity constraint:
\begin{equation}\label{variational_cont}
\widetilde{\gvec}_{\alpha,\beta} 
= \argmin_{\gvec^{(j)}\geq 0}
\left\{\|\mvec-\A \gvec\|_2^2 + \alpha\Regul(\gvec) + \beta\Sep(\gvec) \right\}, 
\end{equation}
where $\alpha,\beta>0$ are regularization parameters, $\gvec^{(j)}\geq 0$ means that the elements of the vector are non-negative numbers and the regularizer $\Regul$ can be any of the standard choices such as the Tikhonov penalty
\begin{equation}\label{tikhonov_penalty}
    \Regul(\gvec) = \|\gvec\|_2^2. 
\end{equation}
%
%
The novelty arises from the term that penalises the Inner Product (IP) of $g^{(1)}$ and $g^{(2)} \in \R^{N^2}$:
\begin{equation}\label{materialsep_penalty}
\Sep(\gvec) = \Sep(\left[\!\!\begin{array}{l}\gvec^{(1)}\\\gvec^{(2)}\end{array}\!\!\right]) := 2\langle \gvec^{(1)},\gvec^{(2)}\rangle
= 2\sum_{i=1}^{N^2} \gvec^{(1)}_{i} \gvec^{(2)}_{i}.
\end{equation}
Together with the non-negativity constraint, $\Sep$ promotes the point-wise separation of the two materials: {\it at each pixel, at least one of the images, $G^{(1)}$ or $G^{(2)}$, needs to have a zero value to make $\Sep$ minimal.} Due to the presence of the inner product, we denote this approach as IP method.

The quadratic program resulting from the application of the novel variational regularization is solved using an Interior Point Method \cite{wright,gondzio_25}; we develop an efficient preconditioner for the normal equations which guarantees a spectrum of the preconditioned matrix to remain independent of the IPM iteration. The numerical experience indicates that this approach allows us to solve the largest problem ($N$=512) in a matter of minutes on a standard laptop.

We demonstrate the feasibility of our new approach to material decomposition with computational experiments. Our specific focus is in low-dose imaging, and therefore we consider imaging with only 65 projection directions. This is roughly one order of magnitude less than in standard tomographic scans. Also, we add simulated noise to the measurements for modelling low-dose exposures. As the baseline method for comparison of reconstruction quality we pick the Joint Total Variation Regularization (JTV), which has been used for dual-energy X-ray CT in \cite{Toivanen2020}. 

We find that under traditional image quality measures, such as square norm error, SSIM or HaarPSI, neither of the two methods show clear superiority over the other. However, when we look at the number of pixels where the materials are correctly identified, our new method outperforms JTV.

For simplicity we restrict here to the case of two materials and two X-ray energies. However, the model readily generalizes to higher numbers of both, as long as there are at least as many energies as there are materials. Moreover, we only consider a two-dimensional slice to be imaged using a one-dimensional linear array X-ray detector. A similar problem could be formulated for 3D objects imaged in cone-beam geometry with a planar X-ray camera; the changes are mathematically straightforward but computationally 
heavy. In an initial feasibility study like this we find it better to stick with numerically straightforward 2D scenarios. 

The paper is organized as follows. In Section \ref{sec:contform} we discuss the continuous theory behind our discrete variational regularization method. Section \ref{sec:IPM} is devoted to presenting an efficient numerical optimization method tailored for finding the minimizer of (\ref{variational_cont}). In Section \ref{sec:JTV} we recall the formulation of the Joint Total Variation regularization approach. In Section \ref{sec:methods} we describe the testing environment and in Section \ref{sec:results} we report numerical results of applying two methods: the new proposed IP regularization and the standard JTV regularization used to analyse several test images. Additionally, we briefly illustrate the behaviour of the preconditioned conjugate gradient, the specialized linear solver applied by the interior point method used to optimize the IP regularization problem. Finally, we summarize our findings in Section \ref{sec:discussion}.

\section{Continuous form of the new regularization functional}\label{sec:contform}

In many inverse problems there is an accurate continuous model for the measurement process. Regularized inversion methods can then be designed and analyzed in infinite-dimensional function spaces \cite{engl1996regularization,mueller2012linear,schuster2012regularization}.

Tomography is a prime example. Given a well-behaving function $f:\Omega\rightarrow \R$, the Radon transform $Rf$ organizes the set of all possible line integrals of $f$:
$$
 Rf(\theta,s) = \int_{x\cdot\theta=s}f(x) dL,
$$
where $\theta\in\R^2$ is a unit vector, $s\in\R$, and $dL$ stands for the one-dimensional Lebesgue measure on the line $x\cdot\theta=s$. Homogenising the molecular scale, we can use a non-negative function $f$ as a model of X-ray attenuation inside a physical object. Further, a logarithmically transformed pixel value in an X-ray camera approximates $Rf(\theta,s)$ with $\theta$ and $s$ determined by the path of the ray hitting the pixel \cite{natterer2001mathematics,mueller2012linear}. 

In practical inverse problems, the unknown needs a finite representation to be used in computational reconstruction. For example, in this work we pixelize $\Omega$, represent $f$ computationally as a function having a constant value on each pixel, and use a pencil-beam model to arrive at the model (\ref{variational_cont}).

Ideally, practical reconstructions can bee seen as discrete approximations of the regularized inversion results described by the continuous theory. This is a great situation as the theorems concerning the continuous model cover all discrete resolutions in one go, providing discretization-invariance for the inversion approach. 

However, the relationship between discrete and continuous inversion models is not always straightforward. For example, in \cite{lassas2004can} two of the authors showed that the total variation prior depends on the discretization in an unexpected and harmful way when used in Bayesian inversion. A discretization-invariant theory was developed using wavelets in \cite{Lassas2009}. Also, the usual assumption of discrete white noise in the practical data requires careful treatment at the infinite-dimensional limit  \cite{kekkonen2016posterior}.

With those caveats in mind, we feel that it is important to provide our new discrete regularization method with a rigorous continuum limit. 

Let  $L^2_+(\Omega)=\{g\in L^2(\Omega):\ g(x)\geq 0\hbox{ a.e.}\}$ and
$g(x)=(g_1(x),g_2(x))\in L^2_+(\Omega)^2$, $\mathcal H$ be a Hilbert space and 
$A:L^2(\Omega)^2\to \mathcal H$ be a bounded linear operator (such as the Radon transform).


We consider the minimization problem
\begin{equation}\label{variational continuous}
\widetilde{g}_{\alpha,\beta} 
= \argmin_{g\in  L^2_+(\Omega)^2}
\left\{\|m-A g\|_{ \mathcal H}^2 + \alpha\Regul(g) + \beta\Sep(g) \right\}, 
\end{equation}
where $\alpha>\beta>0$ are regularization parameters,
and
\begin{equation}\label{tikhonov_penalty continuous}
    \Regul(g) = \int_\Omega(|g_1(x)|^2+|g_2(x)|^2)dx
\end{equation}
and
\begin{equation}\label{materialsep_penalty continuous}
\Sep(g) = \int_\Omega g_1(x)g_2(x)\, dx. 
\end{equation}

Let $U({j,N})\subset \Omega$, $j=1,2,\dots,N$ be  disjoint sets such that $\bigcup_{j=1}^NU({j,N})=\Omega$ and
$\hbox{diam}(U({j,N}))\to 0$ as $N\to \infty$. Let ${\bf 1}_{U({j,N})}(x)=1$ for $x\in U({j,N})$ and
${\bf 1}_{U({j,N})}(x)=0$ for $x\not \in U({j,N})$. In the context of problem (\ref{variational_cont}), the interior of each $U({j,N})$ coincides with the interior of one of the pixels in our discretization of $\Omega$.

Then
$$
\phi_{j,N}(x)=|{U({j,N})}|^{-1/2}{\bf 1}_{U({j,N})}(x),\quad j=1,2,\dots,N,
$$ 
{\color{black}where $|\cdot |$ denotes Lebesgue measure,}
are orthogonal  piecewise constant functions. Let 
$\mathcal P_N\subset L^2(\Omega)$ be the span of
the functions $\phi_{j,N}(x)$, $j=1,2,\dots,N$
and
$$
P_Nu=\sum_{j=1}^N \bra u,\phi_{j,N}\ket_{L^2(\Omega)}\phi_{j,N}
$$
be the orthogonal projector in $L^2(\Omega)$ onto $\mathcal P_N$.
For $g=(g_1,g_2)\in L^2(\Omega)^2$
we denote $P_Ng=(P_Ng_1,P_Ng_2)$.

When $\Omega\subset \R^2$ is the unit square and the interiors of $U({j,N})$ coincide with the interiors of our pixels, the minimizer $\tilde \gvec_{\alpha,\beta}$ defined in (\ref{variational_cont}) corresponds
to a piecewise constant function that solves the minimization problem
\begin{eqnarray}\label{discretized problem 2}
\min_{g\in Y\cap\mathcal P_N^2}\overline F(g),\quad \overline F(g)=\|m-A g\|_{ \mathcal H}^2 + \alpha\Regul(g) + \beta\Sep(g).
\end{eqnarray}
As $\overline F:Y\cap\mathcal P_N^2 \to \R$ is a strictly convex
function and $\mathcal P_N^2$ is a finite dimensional vector space, we see that
$\overline F:Y\cap\mathcal P_N^2 \to \R$ has a unique minimizer.

To study an analogous continuous problem, let
 $F:L^2(\Omega)^2\to \R\cup\{\infty\}$ 
be the function
\begin{eqnarray*}
F(g)&=&\|m-A g\|_{ \mathcal H}^2 + \alpha\Regul(g) + \beta\Sep(g)+\chi_{L^2_+(\Omega)^2}(g),
\end{eqnarray*}
where $\chi_{L^2_+(\Omega)^2}(g)=0$ if $g\in L^2_+(\Omega)^2$ and
$\chi_{L^2_+(\Omega)^2}(g)=\infty$ if $g\not \in L^2_+(\Omega)^2$. 

To study the convergence of the discrete problems,
we define also an auxiliary function
$F_N:L^2(\Omega)^2\to \R\cup\{\infty\}$,
\begin{eqnarray*}
F_N(g)&=&\|m-A P_Ng\|_{ \mathcal H}^2 + \alpha\Regul(g) + \beta\Sep(P_Ng)+\chi_{L^2_+(\Omega)^2}(g).
\end{eqnarray*}

Let $Y=L^2(\Omega)^2$ and $\tau_Y$ be the norm topology of $Y$ and
$\tau_w$ be the weak topology of $Y$.
Consider now a sequence $y_N\in Y$ that converges weakly in $Y$ to $y$.
As $A:Y\to  \mathcal H$ is bounded, and thus $A^*: \mathcal H\to Y$ is bounded, we see
that $Ay_N$ converges weakly in $\mathcal H $ to $Ay$.
Thus, as the norm of a Hilbert space, $\|\cdot\|_{\mathcal H}$, is a weakly lower-semicontinuous function, we see that  $F:Y\to \R\cup\{\infty\}$ is lower-semicontinuous in $(Y,\tau_w)$. 

As $F:Y\to \R\cup\{\infty\}$ is a strictly convex lower-semicontinous function in $(Y,\tau_w)$,
it has a unique minimizer. 
Similarly,  $F_N:Y\to \R\cup\{\infty\}$ 
 has a unique minimizer. 
Moreover, we see that if $g^*_N\in Y$ is a minimizer
of $F_N:Y\to \R\cup\{\infty\}$, then $g^*_N\in\mathcal P_N^2$.
As $F_N(P_Ng)\leq F_N(g)$, we see that
 the minimizer of $F_N$
satisfies $g^*_N\in\mathcal P_N^2$.

Next, we recall the definition of the $\Gamma$-convergence.
Let $(Y,\tau)$ be a topological space and $\{\mathcal{F}_{N}:Y\to [-\infty,\infty],\,N>0\}$ be a 1-parameter family of functionals on $Y$. For $y\in Y$ let $N(x)$ denote the set of all open neighbourhoods $U\subset Y$ of $x$, with respect to the topology $\tau$. If 
\[
\mathcal{F}(x)= \sup_{U\in N(x)}\liminf_{N\to \infty}\inf_{y\in U} \mathcal{F}_{N} (y)= \sup_{U\in N(x)}\limsup_{N\to \infty}\inf_{y\in U} \mathcal{F}_{N}(y),
\]
we say that \emph{$\mathcal{F}_N$ $\Gamma$-converges to $\mathcal{F}$} in $Y$ with respect to  topology $\tau$ as $N\to \infty$. 

\begin{theorem}\label{thm:discretecont}
Let $\alpha>\beta$ and $g^*_N\in Y$ be the minimizers of functions $F_N$ and $g^*\in Y$ be the minimizer of $F$.
Then
\begin{eqnarray}
\lim_{N\to \infty } \|g^*_N-g^*\|_Y=0.
\end{eqnarray}

\end{theorem}


\begin{proof}

Let us first recall the reason why the projectors $P_N$ converge 
strongly to the identity operator in $Y$  as $N\to \infty$. Let $g\in Y$ and $\epsilon>0$.
Then there is a function $g'\in C^1(\overline \Omega)^2$ such that
$\|g-g'\|_Y<\epsilon/4$. Then, $\|P_N(g'-g)\|_Y\leq\|g-g'\|_Y<\epsilon/4$.
Let $M=\|g'\|_{C^1}.$ When $N_0$ is so large that for all $N>N_0$
we have $\hbox{diam}(U({j,N}))<\e/(2M)$,  
we see by considering averages  of $g'$ in the sets
$U(j,N)$ that
$\|g'-P_Ng'\|_Y\leq\e/2$. Thus, for $N>N_0$ we have
$$
\|g-P_Ng\|_Y\leq  \|g-g'\|_Y+\|g'-P_Ng'\|_Y+\|P_Ng'-P_Ng\|_Y<\e.
$$
This shows that the projectors $P_N$ converge 
strongly to $I$ in $Y$  as $N\to \infty$.

\noindent Let $H,H_N:Y\to \R$ be the quadratic functions
\begin{eqnarray*}
H(g)&=&\|m-A g\|_{L^2(\Omega)^2}^2 + \alpha\Regul(g) + \beta\Sep(g)
,\\
H_N(g)&=&\|m-A P_Ng\|_{L^2(\Omega)^2}^2 + \alpha\Regul(g) + \beta\Sep(P_ng).
\end{eqnarray*}
and
 $Q,Q_N:Y\to \R$ be the quadratic forms
\begin{eqnarray*}
Q(g)&=&\|A g\|_{L^2(\Omega)^2}^2 + \alpha\Regul(g) + \beta\Sep(g)
,\\
Q_N(g)&=&\|A P_Ng\|_{L^2(\Omega)^2}^2 + \alpha\Regul(g) + \beta\Sep(P_ng).
\end{eqnarray*}
Observe that for all $g\in Y$ the values  $H_N(g)$ converge to $H(g)$
 as $N\to \infty,$ that is,  $H_N$ converges to $H$ pointwisely in $Y$.
As $H_N$ are convex and uniformly bounded in balls of $Y$, 
 \cite{DalMaso}, Proposition 5.12, implies that $H_N$ $\Gamma$-converges to $H$
 in $(Y,\tau_Y)$ as $N\to \infty.$
 Moreover, $H_N$ converges to $H$ both pointwisely and in the sense of $\Gamma$-convergence,
 and the mapping $g\mapsto \chi_{L^2_+(\Omega)^2}(g)$ is lower-semicontinuous 
 in $(Y,\tau_Y)$. Then  \cite[Propositions 5.9 and 6.25]{DalMaso} imply that
 $F_N:Y\to \R\cup\{\infty\}$ $\Gamma$-converges to $F$ in $(Y,\tau_Y)$ as $N\to \infty.$

As $Q_N(g)\geq  (\alpha-\beta) \|g\|_Y^2$, we see that the family of functions $F_N:Y\to \R\cup\{\infty\}$, $N>0,$ is equicoersive in $(Y,\tau_n)$ by  \cite{DalMaso}, Def. 7.6 and Prop 7.7. By \cite{DalMaso}, Theorem 7.8, we have
$$
F(g^*)=\min_{g\in Y} F(b)=\lim_{N\to \infty}\min_{g\in Y} F_N(b)=\lim_{N\to \infty}F_N(g^*_N).
$$
Observe that as  $g^*_N\in\mathcal P_N^2$, we have
$F_N(g^*_N)=F(g^*_N)$. As $F_N(g)\geq (\alpha-\beta) \|g\|_Y^2$, we see that
$g^*_N$ are uniformly bounded in $Y$. 

To show that $g^*_N$ converges weakly in $Y$ to $g^*$ as
$N\to \infty$, we next assume the opposite. Then, by choosing a subsequence if necessary, we can assume that there is $\e_1>0$ and $y\in Y$ such that 
\beq\label{limits are not same}
|\bra g^*_N,y\ket_Y-\bra g^*,y\ket_Y|>\e_1.
\eeq
By Banach-Alaoglu theorem, by choosing a subsequence if necessary, we can assume that $g^*_N$  converges weakly in $Y$ to some $\tilde g\in Y$.
By \eqref{limits are not same},
$\tilde g\not=g^*$.  

As  $F$ is lower-semicontinuous in $(Y,\tau_w)$,  we have that 
\beq\label{limit of values}
F(\tilde g)\leq \lim_{N\to \infty}F_N(g^*_N)=
F(g^*)=\min_{g\in Y} F(g). 
\eeq
Thus, $F(\tilde g)=F(g^*)$ and $\tilde g$ is a minimizer of  $F:Y\to \R\cup\{\infty\}$.
As the minimizer of $F$ is unique, we have $\tilde g=g^*$ which
is not possible. This shows that $g^*_N$ converges weakly in $Y$ to $g^*$.
This weak convergence, limit \eqref{limit of values} and the fact that
 $g^*_N,g^*\in L^2_+(\Omega)^2$ implies that
\beq\label{limit of values forms}
 \lim_{N\to \infty}Q_N(g^*_N)=Q(g^*). 
\eeq


Observe that $Q:Y\to \R$ is a strongly positive 
quadratic form, that is, $Q(g)\geq (\alpha-\beta)\|g\|_Y^2$ and $\alpha-\beta>0.$
Thus, by \cite{Martin}, Def. 1.1 and property P5 (see also \cite{Hestenes}), the quadratic form $Q:Y\to \R$ is a Legendre form
and it has the property
that if $y_N\to y$ in the weak topology of $Y$
and $Q(y_N)\to Q(y)$ as $N\to \infty,$
then  $y_N\to y$ in the norm topology of $Y$.
Above we have seen that
 $g^*_N$ converges weakly to $g^*$  in $Y$  and 
 the limit 
 \eqref{limit of values forms} holds.
As $Q$ is a Legendre form this implies that
 $g^*_N$ converges in the norm topology $Y$ to $g^*$.

\end{proof}

The message of Theorem \ref{thm:discretecont} is that when we increase the resolution in problems of the form (\ref{variational_cont}), they converge towards a well-defined infinite-dimensional problem. This is a form of discretization-invariance. 

\section{Optimization with preconditioned interior point method (IPM)}\label{sec:IPM}

By combining the use of Tikhonov regularizer (\ref{tikhonov_penalty}) and the Inner Product regularizer (\ref{materialsep_penalty}), which promotes the point-wise separation of two materials, we arrive at the constrained quadratic programming task
\begin{equation}\label{variational_Tikhonov}
\argmin_{\gvec^{(j)}\geq 0}
\left\{\|\mvec-\A \gvec\|_2^2 + \alpha\|\gvec\|_2^2 + \beta\, \gvec^T L \gvec \right\},
\end{equation}
where 
$$
  L=\left[ 
  \begin{array}{cc}
      0 & I \\
      I & 0
  \end{array}\right],
$$
with four blocks of size $N^2{\times}N^2$ each. 


The problem may be written as an explicit quadratic program with inequality (non-negativity) constraints 

\begin{equation}
    \argmin_{\gvec^{(j)}\geq 0} -\mvec^T \A \gvec +\frac{1}{2} \gvec^T(Q_1+Q_2)\gvec
    \label{minproblem}
\end{equation}

where 

\begin{equation}
\small
      Q_1=\left[ 
  \begin{array}{cc}
      c^2_{11}(A^L)^T A^L+c^2_{21}(A^H)^T A^H & c_{11}c_{12}(A^L)^TA^L+c_{21}c_{22}(A^H)^TA^H \\
      c_{11}c_{12}(A^L)^TA^L+c_{21}c_{22}(A^H)^TA^H & c^2_{12}(A^L)^T A^L+c^2_{22}(A^H)^T A^H
  \end{array}\right],
  \label{matrixQ1}
\end{equation}
\begin{equation}
  Q_2=\left[ 
  \begin{array}{cc}
      \alpha I & \beta I \\
      \beta I & \alpha I
  \end{array}\right].
  \label{matrixQ2}
\end{equation}
Notice that $Q=Q_1+Q_2$ can be written as
\begin{equation}
\label{Qkronecker}
{\small Q=\begin{bmatrix} c_{11}^2 & c_{11}c_{12}\\c_{11}c_{12} & c_{12}^2\end{bmatrix}\otimes (A^L)^TA^L+\begin{bmatrix} c_{21}^2 & c_{21}c_{22}\\c_{21}c_{22} & c_{22}^2\end{bmatrix}\otimes (A^H)^TA^H+\begin{bmatrix} \alpha & \beta\\\beta & \alpha\end{bmatrix}\otimes I,}
\end{equation}
where $\otimes$ represents the Kronecker product.

Recall this important property of the Kronecker product:
\begin{lemma}
\label{kroneig}
Given two square matrices $T$ and $Z$, the eigenvalues of the Kronecker product $T \otimes Z$ are given by $t \cdot z$, where $t$ is an eigenvalue of $T$ and $z$ is an eigenvalue of $Z$.
\end{lemma}

\begin{lemma}
\label{lemmaconvex}
If $\alpha\ge\beta$, problem \eqref{minproblem} is convex.
\end{lemma}
\begin{proof}
We just need to show that matrix $Q$ in \eqref{Qkronecker} is positive semi-definite. We know that matrix 
\[\begin{bmatrix} \alpha & \beta\\\beta & \alpha\end{bmatrix}\]
is positive semi-definite if $\alpha\ge\beta$; the other matrices in the right hand side of \eqref{Qkronecker} are always positive semi-definite. Therefore, using Lemma~\ref{kroneig}, $Q$ is the sum of semi-definite matrices and is then positive semi-definite.
\end{proof}
Therefore, in the following we will always assume that $\alpha\ge\beta$.

\subsection{Interior point method formulation}
We decided to solve problem~\eqref{minproblem} using an interior point method \cite{wright,gondzio_25}: these methods are among the most efficient solvers for quadratic programs of large dimensions and can often outperform the more common first order methods in terms of speed of convergence and accuracy. For this problem we aim at reaching large dimensions, and the FISTA method \cite{fista}, already for moderate problem sizes ($N=128$), was not able to match the results of the interior point solver; we thus decided to consider only the latter in this work.

Since the problem (\ref{variational_Tikhonov}) does not involve any linear equality constraints, we can obtain a formulation that is simpler than the general one. In the following, $\evec=(1,1,\dots,1)^T$ and for any vector $\mathbf v$, we define a diagonal matrix $V=\text{diag}(\mathbf v)$. 
To apply an interior point method to (\ref{variational_Tikhonov}), 
we proceed in the usual way and start from adding a logarithmic barrier to form the Lagrangian:
\begin{equation}L(\gvec,\mu)=\frac{1}{2}\gvec^TQ\gvec-\mvec^T\mathcal A\gvec-\mu\sum_{i=1}^{2N^2}\log g_i.\label{lagrangian}\end{equation}
The coefficient $\mu$ is the centrality parameter, which guides the approximations along the central path and which is driven to zero throughout the iterations.
The gradient of~\eqref{lagrangian} is
\[\nabla_gL(\gvec,\mu)=Q\gvec-\mathcal A^T\mvec-\mu G^{-1}\evec.\]
If we define variables $\svec$ as $\mu G^{-1}\evec=S\evec$, then the optimality conditions become
\[\begin{cases}
Q\gvec-\svec=\mathcal A^T\mvec\\
GS\evec=\mu\evec\\
\gvec,\svec>0.
\end{cases}\]

The Newton step $(\Delta\gvec,\Delta\svec)$ for the previous nonlinear system can be found solving
\[\begin{bmatrix} Q & -I\\S & G\end{bmatrix}\begin{bmatrix}\Delta\gvec\\\Delta \svec\end{bmatrix}=\begin{bmatrix}\mathbf{r_1}\\\mathbf{r_2}\end{bmatrix},\]
where $\mathbf{r_1}=\mathcal A^T\mvec-Q\gvec+\svec$ and $\mathbf{r_2}=\sigma\mu\evec-GS\evec$; $\sigma$ is a coefficient that is responsible for the reduction of the parameter $\mu$ \cite{gondzio_25}. 

If we form the normal equations, we obtain the final linear system that we need to solve:
\begin{equation}
\label{normalequations}
(Q+G^{-1}S)\Delta\gvec=\mathbf{r_1}+G^{-1}\mathbf{r_2}.
\end{equation}
We can then retrieve $\Delta\svec$ as
\begin{equation}
\label{deltas}
\Delta\svec=G^{-1}(\mathbf{r_2}-S\Delta\gvec).
\end{equation}

\begin{remark}
Notice that we can use the normal equations without the need to compute the inverse of Q. This would not be possible for a general quadratic program, but here it follows from the fact that we do not have any linear equality constraint.
\end{remark}

At every IPM iteration we need to find the Newton step using \eqref{normalequations}-\eqref{deltas} and calculate the step sizes $\alpha_g$ and $\alpha_s$, so that the new point $(\gvec+\alpha_g\Delta\gvec,\svec+\alpha_s\Delta\svec)$ remains positive. We then update the centrality measure $\mu=\gvec^T\svec/2N^2$ and choose the coefficient $\sigma$ for the next iteration.

In practice, a more sophisticated method is used, which involves predictors and correctors. In particular the predictor, or affine-scaling direction, is computed solving~\eqref{normalequations} with $\sigma=0$. A sequence of correctors is then computed by solving~\eqref{normalequations} with $\mathbf{r_1}=0$ and $\mathbf{r_2}$ chosen in order to improve the centrality of the approximation, by pushing the point towards a symmetric neighbourhood \begin{equation}
N=\{(\gvec,\svec)\mid\gvec>0,\,\svec>0,\,\gamma\mu\le g_js_j\le\mu/\gamma,\,\forall j\}.
\label{neighbourhood}
\end{equation}
This technique, called multiple centrality correctors, has been analyzed in detail in  \cite{gondzio_mcc,colombo_gondzio}.

To stop the IPM iterations, we check the normalized dual residual and the complementarity measure:
\begin{equation}
\frac{\|\mathcal A^T\mvec-Q\gvec+\svec\|}{\|\mathcal A^T\mvec\|}<\texttt{tol}, \quad \mu<\texttt{tol},
\label{IPMstop}
\end{equation}
where $\texttt{tol}$ is the IPM tolerance.

The matrix $Q$ in \eqref{normalequations} is not known explicitly; it is accessible only via matrix-vector products performed using the Radon transform. Hence, to solve the linear system we need to use a matrix free approach; this is done employing conjugate gradient with an appropriate preconditioner.

\subsection{Preconditioner}
The matrix of the system is $Q_1+Q_2+G^{-1}S$, with $Q_1$ given in~\eqref{matrixQ1} and $Q_2$ given in~\eqref{matrixQ2}. $G^{-1}S$ is diagonal, $Q_2$ has a $2\times2$ block structure with diagonal blocks, while the structure of $Q_1$ depends on matrices $(A^L)^TA^L$ and $(A^H)^TA^H$.

Let us analyze an instance where $A^L=A^H=A$. Matrix $A^TA$ is dense in general, but almost all its mass is concentrated in some of its diagonals. Indeed, this can be seen from Figure \ref{fullmatrix}, which shows the magnitude of the elements for the case $N=32$.

In particular, every $N$ diagonals, there is one with larger elements; these elements are almost constant along a specific diagonal, giving matrix $A^TA$ a Toeplitz-like structure. The further away from the diagonal, the smaller the elements become, as can be seen from Figure \ref{diagonaldecay}: here, the mean of the elements along a specific diagonal is plotted against the distance from the main diagonal.

\begin{figure}[h]
\centering
\subfloat[\label{fullmatrix}]{\includegraphics[width=.5\textwidth]{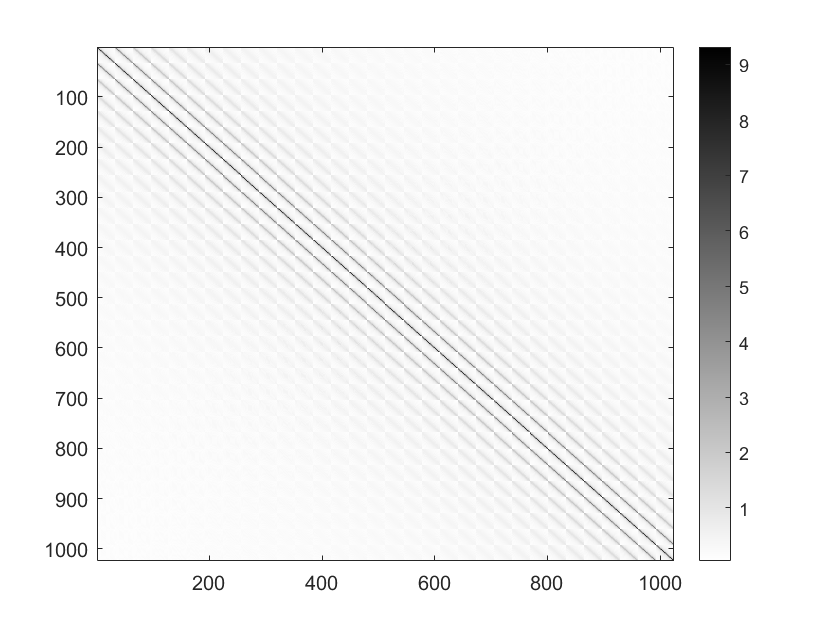}}
\subfloat[\label{diagonaldecay}]{\includegraphics[width=.5\textwidth]{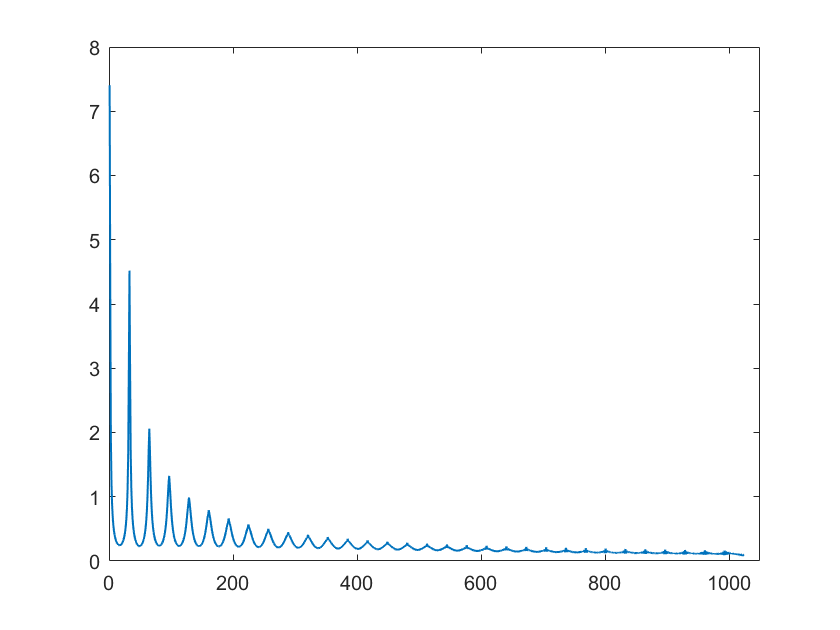}}
\caption{{\bf (a)} Magnitude of the elements of $A^TA$ for $N=32$. {\bf (b)} Magnitude of the mean element along a specific diagonal against the distance from the main diagonal.}
\end{figure}

These facts suggest that it may be possible to approximate matrix $A^TA$ considering only some of the diagonals with large elements. The simplest choice is to use just the main diagonal, in a similar way to what is done in compressed sensing \cite{mf_cs}. Thus, matrix $Q_1$ can be approximated using a $2\times2$ block matrix with diagonal blocks; adding matrix $Q_2$ and $G^{-1}S$ we get the preconditioner:
\begin{equation}
\label{blockdiagprec}
P=\begin{bmatrix}(c_{11}^2+c_{21}^2)\rho I+\alpha I+(G^{-1}S)_1 & (c_{11}c_{12}+c_{21}c_{22})\rho I+\beta I\\ (c_{11}c_{12}+c_{21}c_{22})\rho I+\beta I & (c_{12}^2+c_{22}^2)\rho I+\alpha I+(G^{-1}S)_2\end{bmatrix},
\end{equation}
where we have split the entries of $G^{-1}S$ into the two blocks; $\rho$ is an approximation of the diagonal elements of $A^TA$, obtained through random sampling of this matrix. We will denote the diagonal blocks as $D_{11}$, $D_{12}$ and $D_{22}$ according to their position. This preconditioner is easy to invert: when we need to apply it, we have to solve
\[\begin{bmatrix} D_{11} & D_{12}\\D_{12} & D_{22}\end{bmatrix}\begin{bmatrix}\mathbf{x_1}\\\mathbf{x_2}\end{bmatrix}=\begin{bmatrix}\mathbf{y_1}\\\mathbf{y_2}\end{bmatrix}.\]
This system can be solved forming the Schur complement, which is diagonal:
\[(D_{22}-D_{12}^2D_{11}^{-1})\mathbf{x_2}=\mathbf{y_2}-D_{12}D_{11}^{-1}\mathbf{y_1}\]
and retrieving $\mathbf{x_1}$ from $\mathbf{x_1}=D_{11}^{-1}(\mathbf{y_1}-D_{12}\mathbf{x_2})$.

Notice that most of the terms involved in the preconditioner are constant, while some vary through the IPM iterations, but are immediately available from vectors $\gvec$ and $\svec$. This preconditioner is thus very cheap both to compute and apply.

\begin{remark}
Notice that, if $A^L\ne A^H$, the same preconditioner can be used with a small modification: we just need to approximate both the diagonal of $(A^L)^TA^L$ and $(A^H)^TA^H$ with two different coefficients $\rho^L$ and $\rho^H$.
\end{remark}

In order to use PCG with the proposed preconditioner, we need to show that matrices $Q_1+Q_2+G^{-1}S$ and $P$ are positive definite. 

\begin{lemma}
 If $\alpha\ge\beta$, $M=Q_1+Q_2+G^{-1}S$ and $P$ are symmetric positive definite.
\end{lemma}
\begin{proof}
From Lemma \ref{lemmaconvex} we know that if $\alpha\ge\beta$, matrix $Q$ is positive semi-definite. Matrix $G^{-1}S$ is trivially strictly positive definite, hence $M$ is positive definite.

For $P$, write it as
\[P=\begin{bmatrix} c_{11}^2 & c_{11}c_{12}\\c_{11}c_{12} & c_{12}^2\end{bmatrix}\otimes \rho^L I+\begin{bmatrix} c_{21}^2 & c_{21}c_{22}\\c_{21}c_{22} & c_{22}^2\end{bmatrix}\otimes \rho^H I+\begin{bmatrix} \alpha & \beta\\\beta & \alpha\end{bmatrix}\otimes I+G^{-1}S\]
and proceed in the same way.
\end{proof}

Let us define the matrices
\[F=\begin{bmatrix}f_1 & f_2\\f_2 & f_3\end{bmatrix}\quad K=\begin{bmatrix} \alpha & \beta\\\beta & \alpha\end{bmatrix}\]
where $f_1=c_{11}^2+c_{21}^2$, $f_2=c_{11}c_{12}+c_{21}c_{22}$, $f_3=c_{12}^2+c_{22}^2$.
We can now analyze the spectrum of the preconditioned matrix:
\begin{lemma}\label{lemma_bound}
The eigenvalues of the preconditioned matrix $P^{-1}M$, where $P$ is defined in \eqref{blockdiagprec} and $M=Q_1+Q_2+G^{-1}S$, when $A^L=A^H=A$ satisfy
\[\lambda\in\Bigg[\frac{\alpha-\beta}{\rho\Lambda_F+\alpha+\beta},\frac{\sigma_\text{max}^2(A)\Lambda_F+\alpha+\beta}{\rho\lambda_F+\alpha-\beta}\Bigg],\]
where $\Lambda_F\ge\lambda_F$ are the two eigenvalues of matrix $F$.
\end{lemma}
\begin{proof}
We want to study the generalized eigenvalue problem $M\vvec=\lambda P\vvec$, where 
\[M=F\otimes A^TA+K\otimes I+G^{-1}S,\]
\[P=F\otimes\rho I+K\otimes I+G^{-1}S.\]
Let us fix $\|\vvec\|=1$. The eigenvalues can be expressed as
\[\lambda=\frac{\vvec^TM\vvec}{\vvec^TP\vvec}=\frac{\vvec^T(F\otimes A^TA)\vvec+\vvec^T(K\otimes I)\vvec+\vvec^T(G^{-1}S)\vvec}{\vvec^T(F\otimes\rho I)\vvec+\vvec^T(K\otimes I)\vvec+\vvec^T(G^{-1}S)\vvec}.\]
Let us call the eigenvalues of matrix $F$ as $\Lambda_F>\lambda_F\ge0$, where the last inequality follows from
\begin{align}
f_1f_3-f_2^2&=(c_{11}^2+c_{21}^2)(c_{12}^2+c_{22}^2)-(c_{11}c_{12}+c_{21}c_{22})^2\notag\\
&=c_{11}^2c_{22}^2+c_{21}^2c_{12}^2-2(c_{11}c_{22})(c_{12}c_{21})\notag\\
&=(c_{11}c_{22}-c_{12}c_{21})^2\ge0\notag.
\end{align}
The eigenvalues of $K$ are $\alpha\pm\beta$ and under the assumption $\alpha\ge\beta$, we are sure that this matrix is positive semidefinite.

Using Lemma \ref{kroneig}, we can say that:
\begin{align}
\vvec^T(K\otimes I)\vvec&\in[\alpha-\beta,\alpha+\beta],\notag\\
\vvec^T(F\otimes\rho I)\vvec &\in[\rho\lambda_F,\rho\Lambda_F],\notag\\
\vvec^T(F\otimes A^TA)\vvec&\in[0,\Lambda_F\sigma_\text{max}^2(A)].\notag
\end{align}
Therefore
\begin{equation}\label{lambdaboundupper}
\lambda\le\frac{\sigma_\text{max}^2(A)\Lambda_F+\alpha+\beta+\vvec^T(G^{-1}S)\vvec}{\rho \lambda_F+\alpha-\beta+\vvec^T(G^{-1}S)\vvec},\end{equation}
\begin{equation}\label{lambdaboundlower}
\lambda\ge\frac{\alpha-\beta+\vvec^T(G^{-1}S)\vvec}{\rho\Lambda_F+\alpha+\beta+\vvec^T(G^{-1}S)\vvec}.\end{equation}
Recall the following result: if $A, B, C>0$ then
\[\frac{A+C}{B+C}\ge\frac{A}{B} \Leftrightarrow B\ge A.\]
It is clear that $\rho\Lambda_F+\alpha+\beta\ge\alpha-\beta$ and that $\sigma_\text{max}^2(A)\Lambda_F+\alpha+\beta\ge\rho\lambda_F+\alpha-\beta$, since $\rho$ is the mean eigenvalue of $A^TA$ while $\sigma_\text{max}^2(A)$ the maximum. Thus
\[\lambda\in\Bigg[\frac{\alpha-\beta}{\rho\Lambda_F+\alpha+\beta},\frac{\sigma_\text{max}^2(A)\Lambda_F+\alpha+\beta}{\rho\lambda_F+\alpha-\beta}\Bigg].\]
\end{proof}
\begin{remark}
\label{remark_cg}
Both these bounds do not depend on the IPM iteration. The lower bound depends only on $\alpha$, $\beta$, the coefficients $c_{ij}$ and $\rho$, which does not depend on $N$; hence the lower bound does not depend on $N$. The upper bound, instead, grows as $N$ increases, since the term $\sigma_\text{max}^2(A)$ depends on $N$. Thus, the spectral properties of the preconditioned matrix and the performance of the PCG may deteriorate as N grows.
\end{remark}

A similar result holds in the case $A^L\ne A^H$:
\begin{lemma}
The eigenvalues of the preconditioned matrix $P^{-1}M$, with $A^L\ne A^H$, satisfy
\[\lambda\in\Bigg[\frac{\alpha-\beta}{\Lambda_\rho+\alpha+\beta},\frac{\sigma_\text{max}^2(A^L)\Lambda_{F_L}+\sigma_\text{max}^2(A^H)\Lambda_{F_H}+\alpha+\beta}{\lambda_\rho+\alpha-\beta}\Bigg],\]
where $\lambda_\rho$, $\Lambda_\rho$, $\Lambda_{F_L}$ and $\Lambda_{F_H}$ are defined below.
\end{lemma}
\begin{proof}
In this case, the eigenvalue satisfies
{\small\[\lambda=\frac{\vvec^T(F_L\otimes(A^L)^TA^L)\vvec+\vvec^T(F_H\otimes(A^H)^TA^H)\vvec+\vvec^T(K\otimes I)\vvec+\vvec^T(G^{-1}S)\vvec}{\vvec^T((\rho_L F_L+\rho_H F_H)\otimes I)\vvec+\vvec^T(K\otimes I)\vvec+\vvec^T(G^{-1}S)\vvec},\]}
where
\[F_L=\begin{bmatrix} c_{11}^2 & c_{11}c_{12}\\c_{11}c_{12} & c_{12}^2\end{bmatrix},\quad F_H=\begin{bmatrix} c_{21}^2 & c_{21}c_{22}\\c_{21}c_{22} & c_{22}^2\end{bmatrix}.\]
As before, fix $\|\vvec\|=1$; we can say that
\begin{align} \vvec^T(F_L\otimes(A^L)^TA^L)\vvec&\in[0,\Lambda_{F_L}\sigma_\text{max}^2(A^L)],\notag\\
\vvec^T(F_H\otimes(A^H)^TA^H)\vvec&\in[0,\Lambda_{F_H}\sigma_\text{max}^2(A^H)],\notag\\
\vvec^T((\rho_L F_L+\rho_H F_H)\otimes I)\vvec&\in[\lambda_\rho,\Lambda_\rho]\notag,\end{align}
where we have defined
\[\lambda_\rho=\lambda_\text{min}(\rho_L F_L+\rho_H F_H),\quad\Lambda_\rho=\lambda_\text{max}(\rho_L F_L+\rho_H F_H).\]
Therefore
\[\lambda\le\frac{\sigma_\text{max}^2(A^L)\Lambda_{F_L}+\sigma_\text{max}^2(A^H)\Lambda_{F_H}+\alpha+\beta+\vvec^T(G^{-1}S)\vvec}{\lambda_\rho+\alpha-\beta+\vvec^T(G^{-1}S)\vvec},\]
\[\lambda\ge\frac{\alpha-\beta+\vvec^T(G^{-1}S)\vvec}{\Lambda_\rho+\alpha+\beta+\vvec^T(G^{-1}S)\vvec}.\]
In the same way as before, the final bound becomes
\[\lambda\in\Bigg[\frac{\alpha-\beta}{\Lambda_\rho+\alpha+\beta},\frac{\sigma_\text{max}^2(A^L)\Lambda_{F_L}+\sigma_\text{max}^2(A^H)\Lambda_{F_H}+\alpha+\beta}{\lambda_\rho+\alpha-\beta}\Bigg].\]
\end{proof}

\section{The comparison method: Joint Total Variation (JTV)}\label{sec:JTV}

We have chosen Joint Total Variation (JTV) as a benchmark method for our new Inner Product (IP) regularization method. JTV is a multi-channel joint reconstruction approach where all the unknown images are reconstructed simultaneously by solving one combined inverse problem. Basic (non-joint) TV as a regularizer favors piecewise constant images where the boundary curves separating different constant areas are as short as possible. JTV also promotes piecewise-constantness in each image channel, but additionally {\it requiring that the jump curves in all channels coincide.}


There are many slightly different formulations of the JTV functional in the literature; see 
\cite{yang2009fast,blomgren1998color,ehrhardt2020multi,chung2010numerical,danad2015new}. Total Generalized Variation (TGV) has been used for multi-channel electron microscopy tomography in \cite{huber2019total}.

Let us explain the JTV model used here.

Let $\mathbf{f}^{\,\square}$ be a $N\times N$ matrix, and denote its vertical vector form by $\mathbf{f}\in\R^{N^2}$. Define two $N^2{\times}N^2$ matrices: $L_H$ implementing horizontal differences and $L_V$ vertical differences. The matrix $L_H$ is determined by the formula 
\begin{eqnarray}
  (L_H\mathbf{f})_{\ell} &=& \mathbf{f}^{\,\square}_{k,m+1}-\mathbf{f}^{\,\square}_{k,m},\qquad 1\leq \ell \leq N^2,
\end{eqnarray}
where the row index $k$ and column index $m$ are defined as follows. We write the integer $\ell-1$ in the form
$$
  \ell-1 = (m-1)N + (k-1),
$$
where $0\leq (m-1) < N$ is the quotient and $0\leq (k-1) <N$ is the remainder. Also, we use the convention that $\mathbf{f}^{\,\square}_{k,N+1}=0$ for all $1\leq k\leq N$.
The matrix $L_V$ is determined similarly by the formula 
\begin{eqnarray}
  (L_V\mathbf{f})_{\ell} &=& \mathbf{f}^{\,\square}_{k+1,m}-\mathbf{f}^{\,\square}_{k,m}
\end{eqnarray}
with the convention that $\mathbf{f}^{\,\square}_{N+1,m}=0$ for all $1\leq m\leq N$.

We use JTV for vectors of the form
$$
\gvec = \left[\!\!\begin{array}{l}\gvec^{(1)}\\\gvec^{(2)}\end{array}\!\!\right],
$$
including a non-negativity constraint:
\begin{equation}
\widetilde{\gvec}_{\gamma} 
= \argmin_{\gvec^{(j)}\geq 0}
\left\{\|\mvec-\A \gvec\|_2^2 + \gamma\Regul(\gvec) \right\},
\end{equation} 
where $\gamma>0$ is the regularization parameter. 
The discrete JTV regularizer is 
\begin{eqnarray}\label{discreteJTV}
    \Regul(\gvec) 
    &=&
    \sum_{\ell=1}^{N^2}\left( 
 \left|(L_H\gvec^{(1)})_{\ell}\right|+ \left|(L_V\gvec^{(1)})_{\ell}\right|+
\left|(L_H\gvec^{(2)})_{\ell}\right|+ \left|(L_V\gvec^{(2)})_{\ell}\right|\right).
\end{eqnarray}
In practice we deploy the classical trick of replacing the absolute values in (\ref{discreteJTV}) with a rounded approximate absolute value function $|x|_\kappa =\sqrt{x^2+\kappa}$ with a small parameter $\kappa>0$. This makes the objective functional smooth, allowing straightforward gradient-based minimization. 








\section{Materials and methods}
\label{sec:methods}

We need to find a way to assess the quality of our new method described in the introduction \eqref{variational_cont}. In the spirit of applied inverse problems, we try to evaluate how well the end-users of the algorithm are getting what they want. The main goal is to recover the location of the two different materials in the target, assuming that the materials do not mix. We compare the outcome of our method with the corresponding results from JTV approach to find out if we have reached any improvement.

The new IP method approaches the problem by explicitly representing the two materials as two separate images 
$\gvec^{(1)}$ and $\gvec^{(2)}$ in (\ref{unifiedsystem}), taking into account the energy-dependence of the attenuation coefficients of the materials. The regularized reconstruction determined by (\ref{variational_cont}) gives correspondingly two material images
$$
\widetilde{\gvec}_{\alpha,\beta} = \left[\!\!\begin{array}{l}\widetilde{\gvec}^{(1)}_{\alpha,\beta}\\\\
\widetilde{\gvec}^{(2)}_{\alpha,\beta}\end{array}\!\!\right].
$$
For a known test target we can then check how well the images $\widetilde{\gvec}^{(1)}_{\alpha,\beta}$ and $
\widetilde{\gvec}^{(2)}_{\alpha,\beta}$ match the true locations of the materials. JTV gives us correspondingly two separate material images, which makes comparison straightforward.

We will approximate the quality of our reconstruction method with classical error measures and with pixel error measure, which describes the separation of the materials. We calculate the classical $L_2$-error: 
$$
L_2\text{-error} = \frac{\text{norm}(\text{phantom}(:)-\text{reconstruction}(:))}{\text{norm(phantom(:))}},
$$
 the structural similarity index (SSIM) \cite{wang2004image} and Haar wavelet-based perceptual similarity index (HaarPSI) \cite{reisenhofer2018haar} for both approaches, (JTV and IP) and for both of the materials separately. We calculate these quality measures by comparing the original phantoms with the resulting reconstructions. Same hold for calculating the pixel error. The error calculation protocol needs the following two phases:
\begin{itemize}
    \item[Phase 1.]
    {\bf Choice of optimal regularization parameters.} To allow for a fair comparison between JTV and IP, we need an objective methodology for choosing the regularization parameters. For JTV we look for $\alpha>0$ for which the geometric mean of the relative $L_2$ -errors of the two material images, $E_{\mbox{mean}} = \sqrt{E_1  E_2}$, attains its minimum. For IP method we let $\alpha>0$ vary and take $\beta = 0.8\cdot\alpha$. Then we find the $\alpha$ that minimizes $E_{\mbox{mean}}$.
    \item[Phase 2.]
    {\bf Material characterization error.} The final quality measure for both methods is how well they identify the correct material in each pixel. We assume that we know {\it a priori} the relative amount of each of the two materials. In other words, we know how many pixels should have value one in a material image; the rest of the pixels must be zero. We segment the reconstruction images of both JTV and IP methods by choosing the threshold that yields a binary image with the correct (or most correct) number of pixels with value one.
    
    For example, with HY phantom we have two separate material images, containing only black or white pixels. We use resolution 128$\cdot$128 and can calculate the relative amount of white pixels in the material image 1: 
    $$
    \text{white pixels} = \frac{\text{nnz}(\text{material 1})}{N\cdot N},
    $$
    where nnz-function calculates the number of non-zero pixels. Now when we know the proportion of white pixels, we can define a value of a threshold (tr) so that it sets correct amount of white pixels:
    $$
    \text{Segmented material 1}(\text{reconstruction 1} > tr) = 1;
    \label{eq:segmentation1}
    $$
    and a correct amount of black pixels:
        $$
    \text{Segmented material 1}(\text{reconstruction 1} < tr) = 0;
    $$
    \label{eq:segmentation2}
    in our segmented material image.
    
\end{itemize}

\subsection{Computational parameters in the measurement model}

There were several common settings which we used in the numerical simulations implemented with inner product (IP) method and joint total variation (JTV) method.

 The size of reconstructed images was fixed to be 128x128 pixels in both methods. This quite small resolution was selected for practical reasons: to save memory space and computation time. 
 
 It was important to avoid the inverse crime in the computations, so we added noise and modelling error to the simulations. The relative noise level in both simulations was 0.01. It was added to the measured sinogram by calculating \texttt{noiselevel*max(abs(m(:)))*randn(size(m))}, where m was the simulated sinogram. Besides adding random noise, we avoided the inverse crime by rotating the object 45 degrees, so that the orientation of X-rays changes and interpolation causes small (about 1-2\%) modelling error.

The number of angles in tomographic simulations was chosen to be sparse. Measurement angles were selected between 0 and 180 degrees with constant intervals. (Measurement geometry A, See Figure \ref{fig:arrangement} as an example of imaging geometry). We used parallel-beam geometry and 65 angles for tomographic projections in all of the measurements. 

Attenuation coefficients for high and low energies where selected from NIST-database to simulate the materials of PVC (polyvinyl chloride) and iodine when imaged with 30 kV or 50 kV. Selected values are c11: 1.491 (PVC low energy), c12: 8.561 (Iodine low energy), c21: 0.456 (PVC  high energy), c22: 12.32 (Iodine high energy). See table \ref{tab:att_coeff} for clarity.

\begin{table}[H]
    \centering
    \begin{tabular}{c|c|c}
    Attenuation coefficient & Simulated material & Tube voltage\\ \hline
        1.491  &  PVC    & 30 kV\\
        8.561  &  Iodine & 30 kV\\
        0.456  &  PVC    & 50 kV\\
        12.32  &  Iodine & 50 kV\\
    \end{tabular}
    \caption{Attenuation coefficients selected for simulating the two different materials with low and high tube energies.}
    \label{tab:att_coeff}
\end{table}

\subsection{Phantoms}

We used four different phantoms in all our simulations. First phantom has letters "H" and "Y" in it. This HY phantom represents a piece of plastic, where the letters have been hollowed out, and the holes are filled with iodine. Second "Bone" phantom is an image of a cross section of a bone with bone marrow. Third phantom is a pattern resembling an ancient Egyptian document written in hieroglyph and named as "Egypt" phantom and the last one "Circuit" is an image of an electric circuit.

These four different phantoms pose various challenges to our reconstruction algorithms. We start with a fairly simple HY phantom and gradually add details to raise the standard, so that with the last Circuit phantom we have already many small structures, which are difficult for the algorithms to catch, especially now when data is sparsely collected. For clarity, we show in larger images the results of the more detailed phantoms (Egypt and Circuit).

\begin{figure}[H]
    \centering
    \hspace{0.5cm} \textbf{HY}\hspace{1.7cm} \textbf{Bone}\hspace{1.9cm} \textbf{Egypt}\hspace{1.9cm} \textbf{Circuit}\hspace{0.3cm}
        \includegraphics[width=12cm]{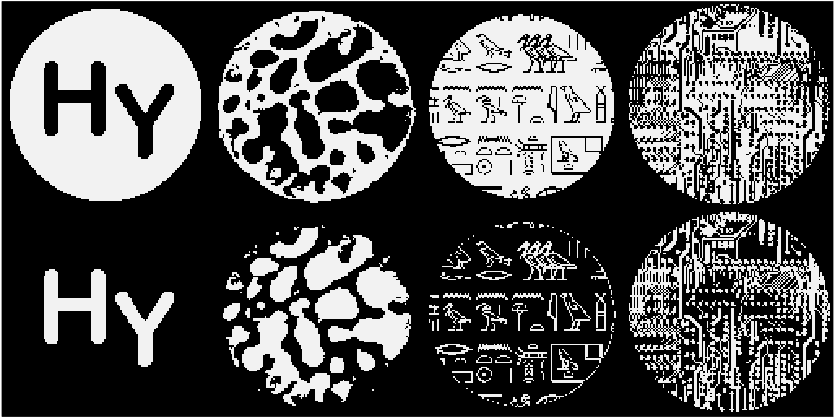}
        \caption{Original phantoms. The four different phantoms which we used in our simulations are shown here in the resolution we actually used. First row shows material one (PVC in these simulations) and second row shows material 2 (iodine in these simulations). These images show the perfect separation of the materials into their own images, so they serve us as a ground truth, where the results of the other methods can be compared.}
               
    \label{fig:phantoms}
\end{figure} 

\section{Results}
\label{sec:results}

In this section we show first reconstruction results and  material decomposition results for our IP method and for standard JTV approach for comparison. We estimate the quality of reconstructions with classical error measures and with material characterization error (misclassified pixels) as described in Section \ref{sec:methods}. We have collected these numerical measures in Table \ref{tab:error_table}. We show also numerical results for assessing the quality of the preconditioner of the IPM method.

\subsection{Reconstruction results of the IP method}

 In IP method we apply Tikhonov regularization and use the inner product  $(g^{(1)})^T g^os$
 We have two regularization parameters $\alpha$ and $\beta$ in this method. The regularization parameter $\alpha$ is chosen by minimizing the mean L2 error in the resulting reconstructions. Parameter $\beta$ adjusts the new regularization term and controls the point-wise separation of the two materials. We fixed $\beta = 0.8\cdot\alpha$ in these simulations. It is important that we always have $\alpha > \beta$. Such a choice prevents from the problem getting non-convex,  which could lead to lengthy computations and an instability of the solution. The stopping criterion for the method is to check the normalized dual residual and the complementarity (duality) gap, see \eqref{IPMstop}. In all our computations the tolerance was set to 1e-8. 

The reconstructions made with IP method for four different phantoms (HY, Bone, Egypt and Circuit), are always in the second column in the following result images  \ref{fig:Hy_recos}, \ref{fig:Bone_recos}, \ref{fig:reconstruction_results_Egypt}, \ref{fig:reconstruction_results_Circuit}. In the first column we have JTV reconstructions for comparison and in the right most column the ground truth. All resulting images have been scaled so that they are in the same scale and thus comparable.

     \begin{figure}[!ht]
\centering
    \begin{subfigure}{0.45\linewidth}
        \hspace{0.2cm} \textbf{JTV}\hspace{0.8cm} \textbf{IP}\hspace{0.3cm} \textbf{\scriptsize Ground\hspace{0.1cm}truth}\hspace{5cm}
        \includegraphics[width=\linewidth]{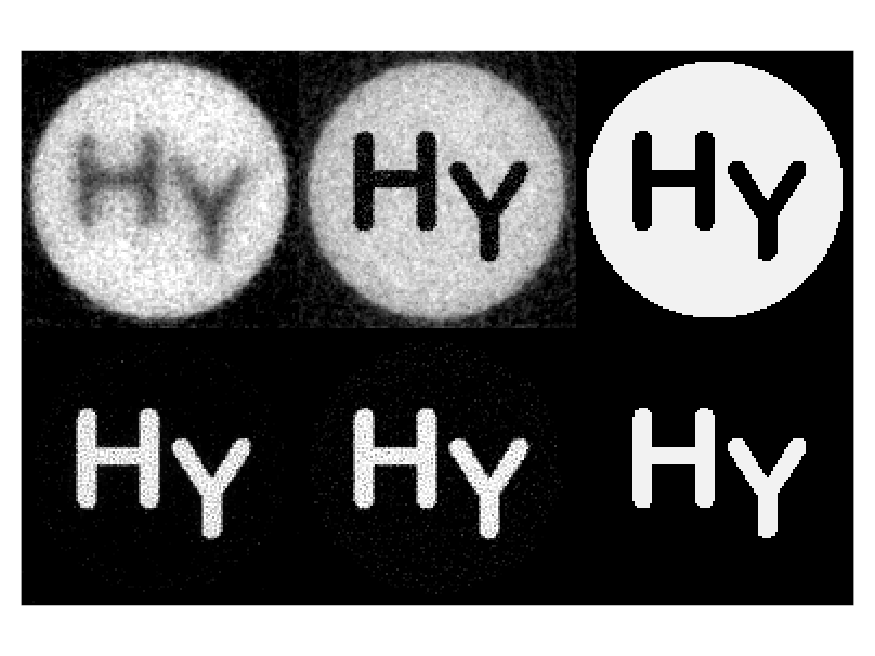}
    \caption{HY}
    \label{fig:Hy_recos}
    \end{subfigure}
\hfil
    \begin{subfigure}{0.45\linewidth}
     \hspace{0.2cm} \textbf{JTV}\hspace{0.8cm} \textbf{IP}\hspace{0.3cm} \textbf{\scriptsize Ground\hspace{0.1cm}truth}\hspace{1.9cm}
        \includegraphics[width=\linewidth]{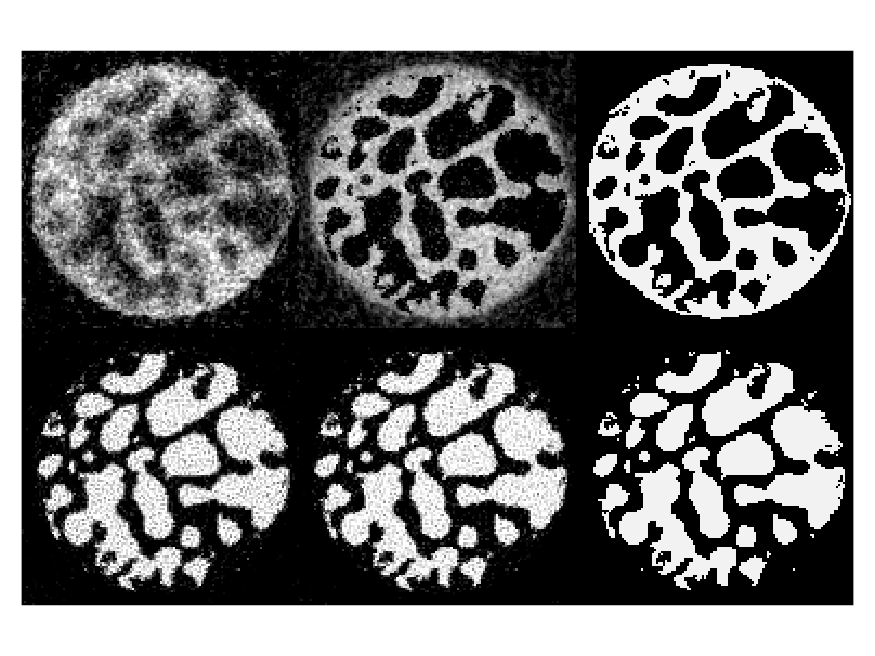}
    \caption{Bone}
       \label{fig:Bone_recos}
    \end{subfigure}
    
\caption{Reconstruction results with JTV and IP regularizations for HY and Bone phantoms. The first row represents material 1 and the second row represents material 2. First column shows JTV reconstructions, second column shows IP-method reconstructions and third column is the ground truth.}
    \label{fig:HY_and_bone_recos}
    \end{figure}

\begin{figure}[H]
    \centering
    \hspace{0.4cm} \textbf{JTV}\hspace{2.5cm} \textbf{IP}\hspace{1.5cm} \textbf{Ground truth}\hspace{1.9cm}
        \includegraphics[width=12cm]{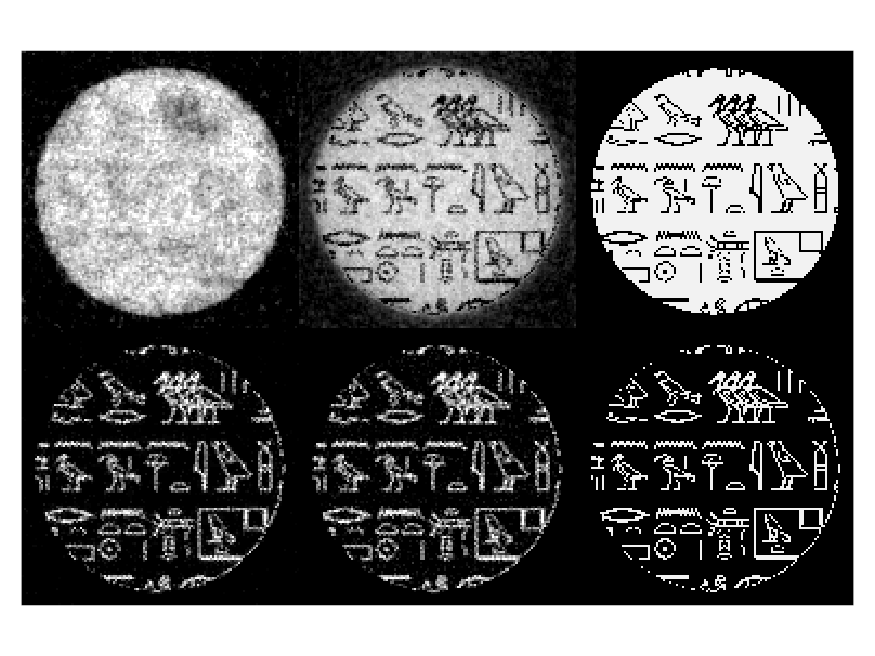}
        \caption{Reconstruction results with JTV and IP regularizations for Egypt phantom. The first row represents material 1 and the second row represents material 2. First column shows JTV reconstructions, second column shows IP reconstructions and third column is the ground truth.}
        
    \label{fig:reconstruction_results_Egypt}
\end{figure} 

\begin{figure}[H]
    \centering

         \hspace{0.4cm} \textbf{JTV}\hspace{2.5cm} \textbf{IP}\hspace{1.5cm} \textbf{Ground truth}\hspace{1.9cm}
        \includegraphics[width=12cm]{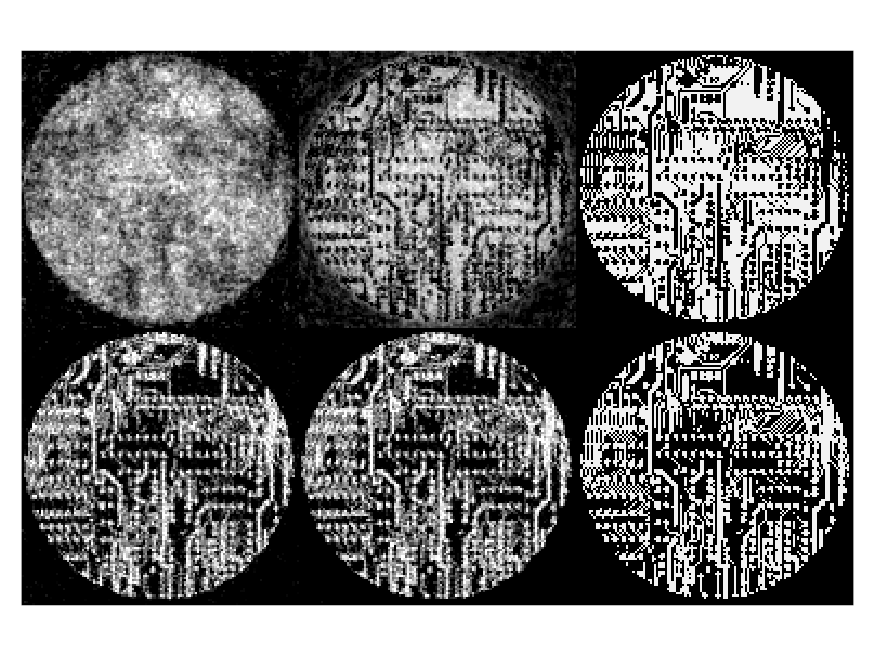}      
        \caption{Reconstruction results with JTV and IP regularizations for electric circuit phantom. The first row represents material 1 and the second row represents material 2. First column shows JTV reconstructions, second column shows IP reconstructions and third column is the ground truth.}
    \label{fig:reconstruction_results_Circuit}
\end{figure} 
\subsection{Reconstruction results with JTV}
 In JTV we use standard Tikhonov regularization for the two image system. Hence we have only one adjustable regularization parameter, $\gamma$. Parameter $\gamma$ was chosen so that it minimizes the mean L2 error in the resulting reconstructions and the value we selected for all cases was $\gamma = 0.001$. The stopping criterion for iterations in JTV was the point where no more progress was made. For all our examples the criterion was achieved in 400 iterations.
 
 The reconstructions made with JTV method for four different phantoms (HY, Bone, Egypt and Circuit) are always in the first column in the reconstruction result images 
 \ref{fig:Hy_recos}, \ref{fig:Bone_recos}, \ref{fig:reconstruction_results_Egypt}, \ref{fig:reconstruction_results_Circuit}. The second column shows IP method reconstructions and the rightmost column shows the actual ground truth.

\begin{table}[H]
\caption{Error calculations for the four phantoms, including the JTV and IP regularizers. Regularization parameters have been adjusted manually to achieve the minimal L2 error. Note that for L2 and {\it misclassifications} a smaller number means better quality, whereas for SSIM and HPSI a greater number means better quality. For JTV we used $\gamma=0.001$. For IP we used $\alpha=150$ and $\beta=120$. For each of the phantoms and fixed quality measure, we have underlined the better result of the two.}
\label{tab:error_table}
\centering
\begin{tabular}{ccccccc}
 Phantom & Method & \phantom{mm}L2\phantom{mm}   & SSIM & HPSI & misclassif.\\
\hline
\\
HY 1 & JTV & 0.30 & 0.23 & 0.21 & 0.05 \\

HY 1 & IP &  \underline{0.27} &  \underline{0.29} & \underline{0.28}  & \underline{0.02} \\
\\
HY 2 & JTV &  \underline{0.27} &  \underline{0.75} & \underline{0.56}  & 0.01\\
HY 2 & IP & 0.28 & 0.60 & 0.53 &0.01 \\
\hline
\\

Bone 1 & JTV & 0.55 & 0.24 & 0.15 & 0.14 \\
Bone 1 & IP & \underline{0.44} & \underline{0.41} & \underline{0.36}  & \underline{0.06} \\
\\              
Bone 2 & JTV & 0.32 & 0.66 & 0.50 & 0.04 \\
Bone 2 & IP & \underline{0.29} & \underline{0.71} & 0.50  & \underline{0.03} \\
\hline
\\
Egypt 1 & JTV & 0.40 &  0.25 & \underline{0.30} & 0.13\\
Egypt 1 & IP & \underline{0.38} & \underline{0.33} & 0.29  & \underline{0.08}\\
\\
Egypt 2 & JTV & 0.62 & 0.69 & 0.56 & 0.06\\
Egypt 2 & IP & \underline{0.61}  & 0.69 & 0.56  & 0.06\\
\hline
\\
Circuit 1 & JTV & 0.62  & 0.17 & \underline{0.30} & 0.28\\
Circuit 1 & IP & \underline{0.56} & \underline{0.32} & 0.28 & \underline{0.18}\\
\\
Circuit 2 & JTV & 0.59  & 0.59 & 0.50 & 0.16\\
Circuit 2 & IP & 0.59 & \underline{0.62} & 0.50 & 0.16\\

\end{tabular}
\end{table}

\subsection{Material decomposition results}

The final quality measure for IP and JTV methods is how well they manage to identify the correct material in each pixel in the reconstructions. Because we work with simulations, we can calculate how many pixels we should have representing material 1 and material 2. With this {\it a priori} knowledge we can adjust the threshold so that it produces the correct number of pixels representing each material. 

The actual ratio of misclassified pixels (divided by the number of all pixels in the image) is listed in the rightmost column of Table \ref{tab:error_table}. We have underlined the better result of the two in the table to make it easier to compare the outcome of the methods.

We show the results of the thresholding also in the following colored segmentation images. Material 1 is represented with yellow color and material 2 with blue color. We hope this makes it easier to qualitatively compare how the methods performed in distinguishing the different materials from each other.

We arranged the colored segmentation images as a grid, where column represents the method and row represents the outcome. The first row in the segmentation result shows both materials in the same image. The second and third row show the materials separately in their own images: Material 1 in the second row and material 2 in the third row. Columns in all images have been organized so that JTV approach is always in the first column, IP regularization is in the second column and the actual ground truth is in the rightmost column. The ground truth represents the ideal situation where the classification of the materials has succeeded perfectly.

     \begin{figure}[!ht]
\centering
    \begin{subfigure}{0.45\linewidth}
        \hspace{0.2cm} \textbf{JTV}\hspace{0.8cm} \textbf{IP}\hspace{0.3cm} \textbf{\scriptsize Ground\hspace{0.1cm}truth}\hspace{5cm}
        \includegraphics[width=\linewidth]{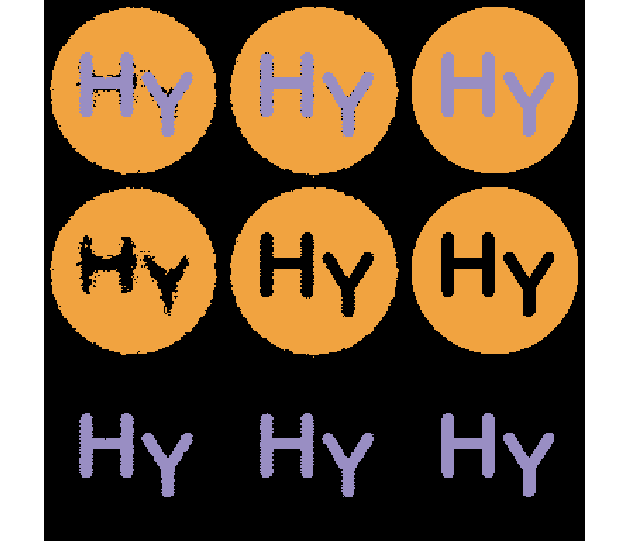}
    \caption{HY}
    \label{fig:seg_results}
    \end{subfigure}
\hfil
    \begin{subfigure}{0.45\linewidth}
     \hspace{0.2cm} \textbf{JTV}\hspace{0.8cm} \textbf{IP}\hspace{0.3cm} \textbf{Ground truth}\hspace{1.9cm}
        \includegraphics[width=\linewidth]{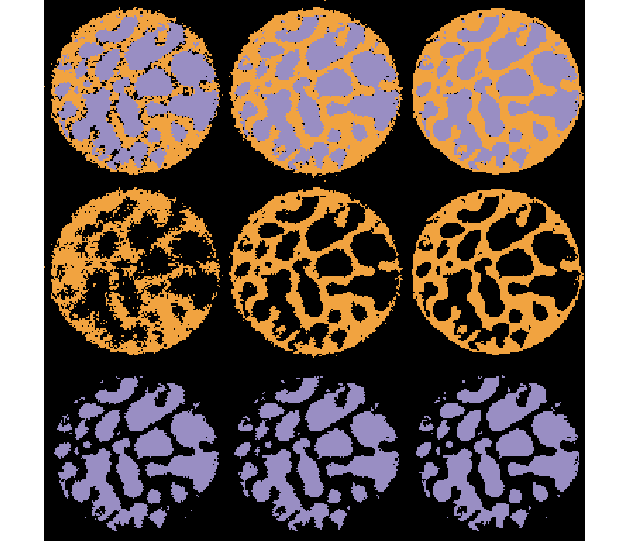}
    \caption{Bone}
       \label{fig:Bone_segmentations}
    \end{subfigure}
    
\caption{Segmentation results for HY and bone phantoms. The first row shows both materials of the phantom together, the second row shows only material 1 and the third row shows only material 2. The first column shows JTV segmentations, the second column shows IP segmentations and the third column is the ground truth.}
    \label{fig:HY_and_bone_segmentations}
    \end{figure}

\begin{figure}
    \centering
       
         \hspace{0.4cm} \textbf{JTV}\hspace{2.5cm} \textbf{IP}\hspace{1.5cm} \textbf{Ground truth}\hspace{1.9cm}
        \includegraphics[width=10cm]{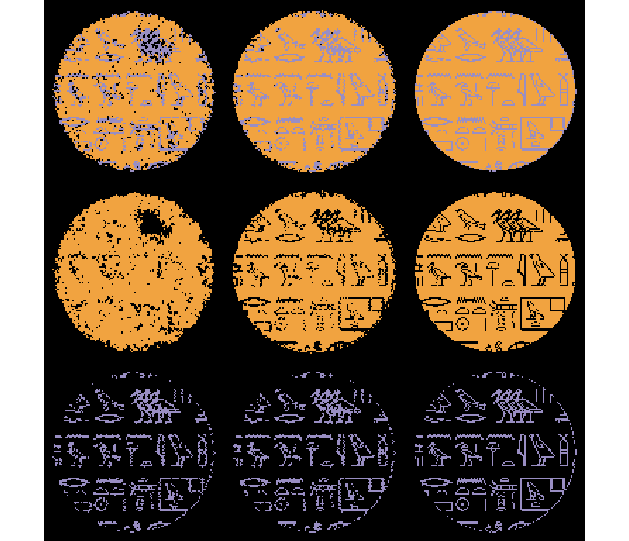}
        \caption{Segmentation results for Egypt phantom. The first row shows both materials of the phantom together, the second row shows only material 1 and the third row shows only material 2. The first column shows JTV segmentations, the second column shows IP segmentations and the third column is the ground truth.}
    \label{fig:segmentation_results_Egypt}
\end{figure} 

\begin{figure}
    \centering
      
         \hspace{0.4cm} \textbf{JTV}\hspace{2.5cm} \textbf{IP}\hspace{1.5cm} \textbf{Ground truth}\hspace{1.9cm}
        \includegraphics[width=10cm]{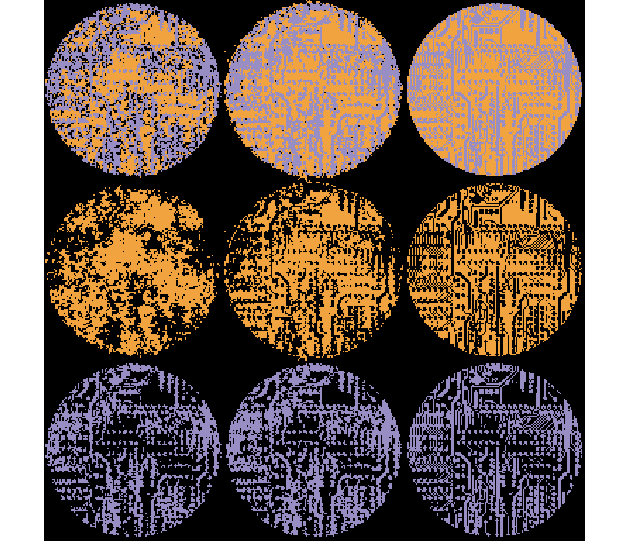}
        \caption{Segmentation results for Circuit phantom. The first row shows both materials of the phantom together, the second row shows only material 1 and the third row shows only material 2. The first column shows JTV segmentations, the second column shows IP segmentations and the third column is the ground truth.}
    \label{fig:segmentation_results_Circuit}
\end{figure} 

\subsection{Numerical effect of preconditioning}

In this section we present the results which provide an insight into the behaviour of optimization technique employed to solve the IP segmentation problem (\ref{variational_Tikhonov}). We briefly discuss the performance of interior point method applied to solve the underlying convex quadratic programming problem and focus on illustrating the behaviour of the preconditioned conjugate gradient algorithm applied to normal equations (\ref{normalequations}) arising in IPM. 

We start by showing in Figure~\ref{spectrumiteration} the eigenvalues of the normal equations, with and without preconditioner (\ref{blockdiagprec}), for the problem with $N=32$. It is clear that the spectrum of the preconditioned matrix is bounded independently of the IPM iteration, which is what we were expecting according to Lemma~\ref{lemma_bound}.





\begin{figure}[H]
\caption{Eigenvalues of the normal equations with and without preconditioner for $N=32$, $\alpha=500$, $\beta=250$}
\label{spectrumiteration}
\centering
\includegraphics[width=.9\textwidth]{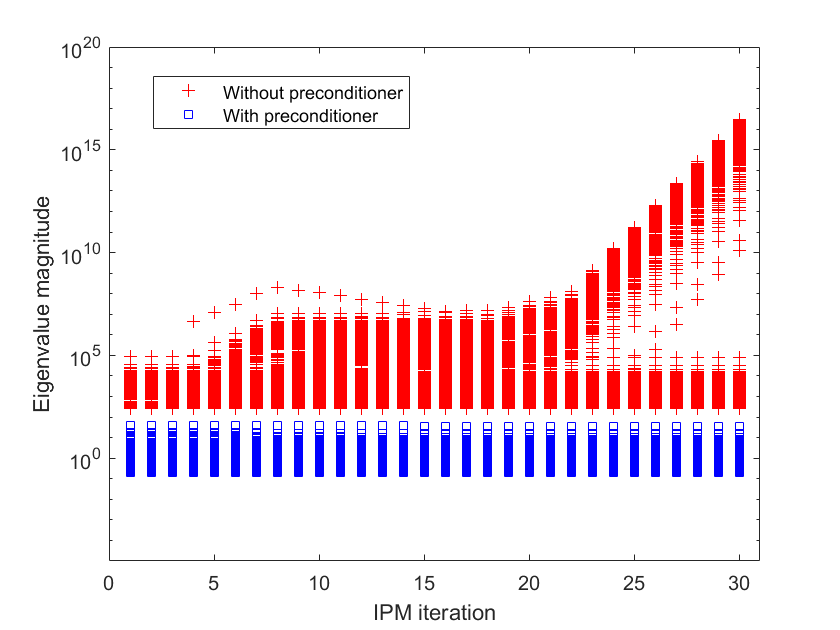}
\end{figure}

Next, we show in Table~\ref{ipm_results} the results in terms of IPM iterations, PCG iterations and computational time, for various values of N. The IPM tolerance in~\eqref{IPMstop} is set to  $10^{-8}$; we employed 3 centrality correctors with a symmetric neighbourhood~\eqref{neighbourhood} with parameter $\gamma=0.2$. The default PCG tolerance is $10^{-6}$, but we also employ an early termination strategy, based on the estimate of the IPM convergence indicators throughout the CG iterations (see \cite{cg_termination} for more details).



\begin{table}[H]
\caption{Results with $\alpha=500$, $\beta=250$.}
\label{ipm_results}
\centering
\begin{tabular}{rr|rrr}
\toprule
$N$ & Dimension & IPM iter & PCG iter & Time (s) \\
\midrule
 32 &   2,048 & 19 & 1,038 &   2.25\\
 64 &   8,192 & 24 & 1,484 &   7.90\\
128 &  32,768 & 25 & 1,986 &  32.69\\
256 & 131,072 & 28 & 2,678 & 157.79\\
512 & 524,288 & 34 & 3,772 & 881.90\\
\bottomrule
\end{tabular}
\end{table}

As we were expecting from Remark~\ref{remark_cg}, we can see that the number of CG iterations per IPM iteration grows slowly as N increases. However, such behaviour is acceptable, and the computational time in the case of $N=512$ is still reasonable.

\subsubsection{Effect of the regularization}
We also show some results that underline the effect of the newly added penalty term~\eqref{materialsep_penalty}. We expect from this regularizer to create a separation in the vectors $\gvec^{(1)}$ and $\gvec^{(2)}$, i.e.\ we expect the scalar product $\gvec^{(1),T}\gvec^{(2)}$ to be pushed close to zero. We performed some tests with different values of $\beta$ and a fixed value $\alpha=500$, in the case $N=64$.

Table \ref{betaeffect} shows the number of elements of the component-wise products of $\gvec^{(1)}$ and $\gvec^{(2)}$ that are smaller than $10^{-6}$, and the average value of the same product, i.e.\ $(\gvec^{(1),T}\gvec^{(2)})/N^2$. We can see that as $\beta$ is increased, the number of small elements grows and the average product decreases, confirming the effect that we expected.

\begin{table}[H]
\caption{Number of small elements and average product of $\gvec^{(1)}$ and $\gvec^{(2)}$ for different values of $\beta$; $\alpha=500$, $N=64$.}
\label{betaeffect}
\centering
\begin{tabular}{rcc}
\toprule
$\beta$ & small elements & $(\gvec^{(1),T}\gvec^{(2)})/N^2$\\
\midrule
 50 & 1056 & 4.86E3\\
100 & 1091 & 4.07E3\\
150 & 1123 & 3.17E3\\
200 & 1161 & 2.36E3\\
250 & 1607 & 1.54E3\\
300 & 2075 & 1.23E3\\
350 & 2210 & 1.07E3\\
400 & 2412 & 0.93E3\\
450 & 2581 & 0.83E3\\
\bottomrule
\end{tabular}
\end{table}

Figure \ref{g1g2product} shows the elements of the component-wise products of $\gvec^{(1)}$ and $\gvec^{(2)}$, sorted according to their magnitude, in the case $\beta=50$ and $\beta=450$. The number of small elements is substantially larger in the latter case, confirming what we expected. 

\begin{figure}[H]
\caption{Magnitude of the elements of the component-wise products of $\gvec^{(1)}$ and $\gvec^{(2)}$.}
\label{g1g2product}
\centering
\includegraphics[width=.7\textwidth]{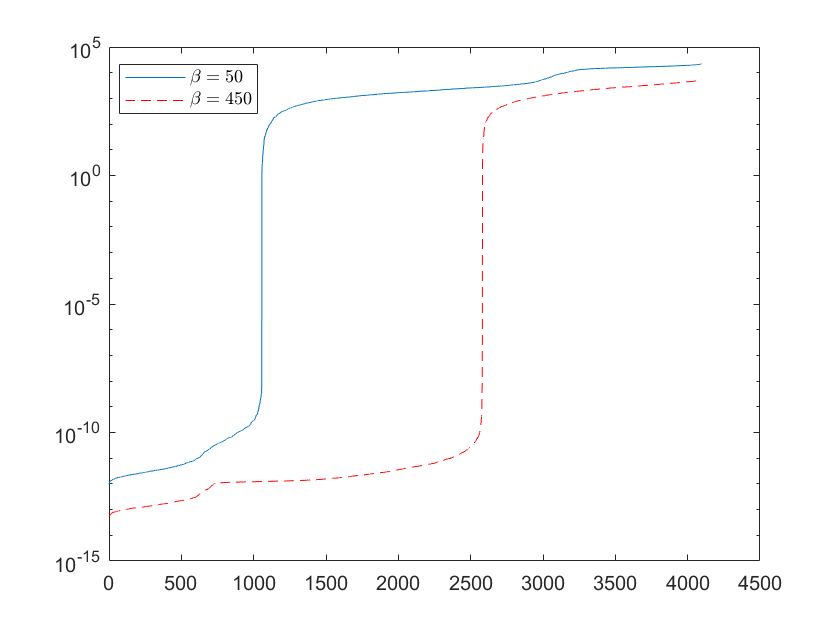}
\end{figure}

\section{Discussion}\label{sec:discussion}

When we compare the color segmentation results achieved with the two approaches JTV and IP, we can easily see that IP delivers a segmentation with fewer misclassified pixels. Hence IP method produces more accurate separation of the materials. The actual ratio of misclassified pixels compared to all pixels is listed in Table \ref{tab:error_table}. This numerical evidence suggests that IP is consistently better in pixel misclassification quality measure which is a crucial quality indicator for the application we have in mind. IP is also a frequent winner 
(although less consistent) for the remaining quality measures. To be precise, JTV is better than IP only in 1 case out of 8 on L2 measure, only in 1 case out of 8 on SSIM and in 3 cases out of 8 on HPSI. 

Furthermore, it seems that JTV always produces visibly worse reconstruction of Material 1 image than that of Material 2. This could probably be alleviated by a different weighting of the gradient components. However, in the comparisons in this paper we used both methods in their basic forms, as both can undoubtedly be improved by tweaking various parameters. 

One such tweak would be a smarter thresholding, taking into account both material reconstructions and the piece of {\it a priori} knowledge that each pixel contains exactly one type of material.

The natural next step is to test the new method with two-dimensional X-ray images recorded of a three-dimensional object, using voxels instead of pixels for computational discretization. 

\bibliographystyle{plain}
\bibliography{mybibliography}

\end{document}